\documentclass[12pt,a4paper,reqno]{amsart}
\usepackage[utf8]{inputenc}
\usepackage{amsmath}
\usepackage{amsfonts}
\usepackage{amssymb}
\usepackage[all]{xy}
\usepackage{xfrac}
\usepackage{enumerate}

\usepackage[english]{babel}

\newtheorem{theorem}{Theorem}[section]

\newtheorem{lemma}[theorem]{Lemma}
\newtheorem{definition}[theorem]{Definition}
\newtheorem{proposition}[theorem]{Proposition}
\newtheorem{remark}[theorem]{Remark}
\newtheorem{example}[theorem]{Example}

\newcommand{\C}{\mathbb{C}}
\DeclareMathOperator{\cod}{cod}
\DeclareMathOperator{\dom}{dom}

\DeclareMathOperator{\coker}{coker}

\author{A. P. Garr\~{a}o}
\address[Ana Paula Garr\~{a}o]{Faculdade de Ci\^{e}ncias e Tecnologia, Universidade dos A\c cores, Ponta Delgada,
Portugal}
\thanks{ }

\email{ana.po.garrao@uac.pt}

\author{N. Martins-Ferreira}
\address[Nelson Martins-Ferreira]{Instituto Polit\'{e}cnico de Leiria, Leiria, Portugal}
\thanks{ }
\email{martins.ferreira@ipleiria.pt}

\author{M. Raposo}
\address[Margarida Raposo]{Faculdade de Ci\^{e}ncias e Tecnologia, Universidade dos A\c cores, Ponta Delgada, Portugal}
\thanks{ }

\email{margarida.js.raposo@uac.pt}

\author{M. Sobral}
\address[Manuela Sobral]{CMUC and Departamento de
Matem\'atica, University of Coimbra, 3001--501 Coimbra,
Portugal}
\thanks{ }

\email{sobral@mat.uc.pt}

\title[]{Cancellative conjugation semigroups and monoids}


\begin{document}

\begin{abstract}
We show that the category of cancellative conjugation semigroups is weakly Mal'tsev and give a characterization of all admissible diagrams there.
In the category of cancellative conjugation monoids we describe, for Schreier split epimorphisms with codomain $B$ and kernel $X$, all morphisms $h:X\rightarrow B$ which induce a reflexive graph, an internal category or an internal groupoid. We describe Schreier split epimorphisms in terms of external actions and consider the notions of precrossed semimodule, crossed semimodule and crossed module in the context of cancellative conjugation monoids. In this category we prove that a relative version of the so-called ``Smith is Huq" condition for Schreier split epimorphisms holds as well as other relative conditions.
\end{abstract}

\maketitle

\today

\section{Introduction}

In \cite{NMF.08a} the concept of weakly Mal'tsev category was introduced to provide a simple axiomatic context where the internal categories and
the internal groupoids are particularly simple to describe. The established notion of Mal'tsev category (\cite{CLP.91}) is too restrictive for this purpose
since there the two notions coincide. Amongst the categories that are weakly Mal'tsev but not Mal'tsev are the categories of distributive lattices
(\cite{NMF.12}) and the category of cancellative commutative monoids. In this paper we introduce another class of examples of such categories,
that includes the latter,
characterize there all the admissible diagrams and describe some internal structures. The admissibility of  certain type of diagrams is used to
go from local to global in a sense we make precise below.

We introduce the category of conjugation semigroups which can be seen as an abstraction of conjugation of complex numbers or of quaternions.
A conjugation semigroup $(S, \cdot, \overline{(\; )})$ is a semigroup $(S, \cdot)$
equipped with a unary operation
$\overline{(\; )}\colon S \rightarrow S$ satisfying the following identities:
 $x  \overline{x}= \overline{x} x$,
 $x  \overline{y}  y = y  \overline{y}  x$,
and  $ \overline{x  y}= \overline{y} \,  \overline{x}$.
The quasivariety of cancellative conjugation semigroups is a weakly Mal'tsev category and we present a characterization of all admissible diagrams in the sense of \cite{NMF.08}. We observe that, since they satisfy the Ore condition (\cite{OO.31}), all cancellative conjugation semigroups can be embedded in a group.

In the category of cancellative conjugation monoids we describe, for Schreier split epimorphisms  with codomain $B$ and kernel $X$
(a notion first introduced in \cite{FMS.13} for monoids with operations),
$$\xymatrix{ X \ar@<-2pt>[r]_k & A \ar@<-2pt>@{.>}[l]_{q_f}
\ar@<-2pt>[r]_f & B \ar@<-2pt>[l]_r },$$
all morphisms $h\colon X\rightarrow B$ which induce a reflexive graph, an internal category or an internal groupoid. That is, all morphisms
$h\colon X\rightarrow B$ that induce a morphism $\tilde{h}\colon A\rightarrow B$, with $\tilde{h} k = h$, such that
$\xymatrix{ A \ar@<-4pt>[r]_{\tilde{h}} \ar@<4pt>[r]^f & B \ar[l]|r }$
gives rise to a reflexive graph, an internal category or an internal groupoid.

We recall that, in a semi-abelian category $\mathbb{C}$ (\cite{JMT.02}), there is an equivalence between the category of split epimorphisms  \text{SplitEpi($\mathbb{C}$)}  and the category of internal actions \text{Act($\mathbb{C}$)} as introduced in \cite{BJ.98}. This equivalence induces further equivalences at the level of reflexive graphs and internal categories, as proved by  G. Janelidze (\cite{GJ.03}),  with the categories of precrossed modules and of crossed modules, respectively, in $\mathbb{C}$, as displayed in the following diagram.

\begin{equation}\nonumber
\xymatrix{\text{SplitEpi($\mathbb{C}$)} & \simeq & \text{Act($\mathbb{C}$)}\\
\text{RGraph($\mathbb{C}$)}\ar[u] & \simeq & \text{PrecrossedMod($\mathbb{C}$)}\ar[u]\\
\text{Cat($\mathbb{C}$)}\ar@{^{(}->}[u] & \simeq & \text{CrossedMod($\mathbb{C}$)}\ar@{^{(}->}[u]
}
\end{equation}

When $\mathbb{C}$ is the category of groups, all the categories on the right column correspond to the classical definitions of categories of group actions, precrossed and crossed modules, respectively, where the classical definition of  an action of a group  $B$ on a group $X$ is a group homomorphism $\varphi\colon{B\to \text{Aut($ X$)}}$. For a given action $\varphi$, denoted as $\varphi(b)(x)=b\cdot x$, a morphism $h\colon{X\to B}$  is part of a precrossed module structure $(B, X, \varphi, h)$ when for all $b\in B$ and $x\in X$, \begin{equation}
h(b\cdot x)=b+h(x)-b
\end{equation}
and $(B, X, \varphi, h)$  is a crossed module if, in addition,
\begin{equation}
h(x)\cdot y=x+y-x
\end{equation}
for every $x,y\in X$.

If  $\xymatrix{ X \ar@<0pt>[r]^k  & A
\ar@<-2pt>[r]_f & B \ar@<-2pt>[l]_r }$ is, up to isomorphism,  the split epimorphism associated with the action $\varphi$, then $h\colon{X\to B}$ induces a reflexive graph $\xymatrix{ A \ar@<-4pt>[r]_{\tilde{h}} \ar@<4pt>[r]^f & B \ar[l]|r }$
 if and only if $(B, X, \varphi, h)$ is a precrossed module. Moreover, $h$ induces an internal category if and only if $(B, X, \varphi, h)$ is a crossed module. Furthermore, in the category of groups, and more generally in any semi-abelian category, internal categories and internal groupoids are the same (see for example \cite{CLP.91} where this result is even proved for the more general case of Mal'tsev categories).

 In the context of  monoids (and more generally of monoids with operations) these equivalences do not necessarily hold. In these categories internal actions do not coincide, in general, with what we call external actions, that correspond to the classical notion of monoid action: an external action of a monoid $B$ on a monoid $X$ is a monoid homomorphism $\varphi\colon{B\to \text{End($X$)}}$ (See \cite{FMS.14}).

  A systematic study of the split epimorphisms that correspond to external actions of monoids and of monoids with operations, that were called Schreier split epimorphisms, was done in \cite{BFMS.14,BFMS.16,FMS.13} and it leads to the conclusion that these split epimorphisms still have many of the properties of the category of all split epimorphisms in a semi-abelian context.

 In the category of cancellative conjugation monoids, here simply denoted as $\mathcal{M}$, we consider an external action $\varphi$, represented as a Schreier split epimorphism $$\xymatrix{ X \ar@<-2pt>[r]_k & A \ar@<-2pt>@{.>}[l]_{q}
\ar@<-2pt>[r]_f & B \ar@<-2pt>[l]_r },$$ and a morphism $h\colon{X\to B}$. The pair $(\varphi,h)$ induces a reflexive graph, an internal category, or an internal groupoid (indeed, internal categories and internal groupoids are no longer the same) if and only if the following conditions are satisfied:
\begin{eqnarray}
h(b\cdot x)+b=b+h(x), & \text{for all $x\in X$, $b\in B$}\label{eq:C1}\\
h(y)\cdot x+y=y+x, &  \text{for all $x,y\in X$}\label{eq:C2}\\
\text{$X$ is a group and $-\overline{x}=\overline{(-x)}$}, & \text{for all $x\in X$.}\label{eq:C3}
\end{eqnarray}
More precisely, when condition (\ref{eq:C1}) holds, $(\varphi,h)$ induces a reflexive graph and we call it a precrossed semimodule. When conditions (\ref{eq:C1}) and $(\ref{eq:C2})$ hold, $(\varphi,h)$ induces an internal category and we say that it is a crossed semimodule. When conditions $(\ref{eq:C1})$, $(\ref{eq:C2})$ and $(\ref{eq:C3})$ hold, then $(\varphi,h)$ induces an internal groupoid and we call it a crossed module.

In other words we have:
\begin{equation}\nonumber
\xymatrix{\text{SchreierSplitEpi($\mathcal{M}$)} & \simeq & \text{ExtAct($\mathcal{M}$)}\\
\text{SchreierRGraph($\mathcal{M}$)}\ar[u] & \simeq & \text{PrecrossedSemiMod($\mathcal{M}$)}\ar[u]\\
\text{SchreierCat($\mathcal{M}$)}\ar@{^{(}->}[u] & \simeq & \text{CrossedSemiMod($\mathcal{M}$)}\ar@{^{(}->}[u]\\
\text{SchreierGrpd($\mathcal{M}$)}\ar@{^{(}->}[u] & \simeq & \text{CrossedMod($\mathcal{M}$)}\ar@{^{(}->}[u]
}
\end{equation}

\hfill

The external actions in $\mathcal{M}$ are described in Section  \ref{sec:5}.

\hfill

In the context of semi-abelian categories, it was made clear (e.g. \cite{FL.12}) that the so-called ``Smith is Huq" condition is precisely the one that allows a simple description of internal structures. The existence of this categorical equivalence in the context of cancellative conjugation monoids is also a consequence of a relative version of the ``Smith is Huq" condition that we describe here (Section \ref{sec:local to global}).

In order to better understand this phenomenon we observe the following. Proposition \ref{prop:5.4} shows that if condition $(\ref{eq:C1})$ holds true for a pair $(\varphi,h)$, then condition $(\ref{eq:C2})$ is equivalent to the existence of a multiplication on the reflexive graph induced by $(\varphi,h)$. The proof given in Proposition \ref{prop:5.4} follows a constructive method and it can be generalized to any weakly Mal'tsev category. However, in the particular case of cancellative conjugation monoids there is a deeper reason for that result to hold. This phenomenon is thoroughly examined in Section \ref{sec:local to global} where, among other things, we observe that condition $(\ref{eq:C2})$ follows from a specialised version of the ``Smith is Huq" condition (Proposition \ref{prop:7.4}).


\vspace*{0.3cm}
Throughout, for simplicity of exposition, we will use the additive notation for monoids and semigroups, although we do not assume commutativity.

\section{Preliminaries}

We recall here definitions and basic properties that will be used throughout.

\subsection{Weakly Mal'tsev categories}

In a category with pullbacks of split epimorphisms along split epimorphisms we consider the following diagram

$$
\vcenter{
\xymatrix@=4em{ A \times_B C \ar@<-.5ex>[d]_{\pi_1}
\ar@<-.5ex>[r]_-{\pi_2} & C \ar@<-.5ex>[d]_g \ar@<-.5ex>[l]_-{e_2} \\
A \ar@<-.5ex>[u]_{e_1} \ar@<-.5ex>[r]_f & B \ar@<-.5ex>[u]_s
\ar@<-.5ex>[l]_r }
}
$$
where  $f r=1_{B}=g s$, $(\pi_1,\pi_2)$ is a pullback of $(f,g)$ and $e_1= \langle 1_A, sf\rangle$, $e_2 = \langle  rg, 1_C\rangle$ are the morphisms induced by the universal property of the pullback. Then $e_1r=e_2s$.

Any diagram

\begin{equation}\label{admissibility diagram}
\vcenter{\xymatrix@=4em{A \ar@<.5ex>[r]^-{f} \ar[rd]_-{\alpha} & B
\ar@<.5ex>[l]^-{r}
\ar@<-.5ex>[r]_-{s}
\ar[d]^-{\beta} & C \ar@<-.5ex>[l]_-{g} \ar[ld]^-{\gamma}\\
& D}}
\end{equation}
with $f r=1_{B}=g s$ and $\alpha r=\beta=\gamma s$ will be called an \emph{admissibility diagram}. It induces a diagram

$$
\vcenter{\xymatrix@!0@=3em{ & C \ar@<.5ex>[ld]^-{e_2} \ar@<-.5ex>[rd]_-{g}
\ar@/^/[rrrd]^-{\gamma} \\
A\times_{B}C \ar@<.5ex>[ru]^-{\pi_2}
\ar@<-.5ex>[rd]_-{\pi_1} && B \ar@<.5ex>[ld]^-{r} \ar@<-.5ex>[lu]_-{s}
 \ar[rr]|-{\beta} && D.\\
& A \ar@<.5ex>[ru]^-{f} \ar@<-.5ex>[lu]_-{e_1} \ar@/_/[urrr]_-{\alpha}}}
$$
The existence of a unique morphism $\varphi\colon{A\times_{B}C \to D}$ such that $\varphi e_1=\alpha$ and $\varphi e_2=\gamma$ is a way to describe relevant situations and results in categorical algebra as we mention next.

\begin{definition}(\cite{NMF.08a}). The triple $(\alpha,\beta,\gamma)$ is admissible with respect to $(f,r,g,s)$ if there exists a unique morphism  $\varphi\colon{A\times_{B}C \to D}$ such that $\varphi e_1=\alpha$ and $\varphi e_2=\gamma$. Then we say that the diagram $(\ref{admissibility diagram})$ is admissible.
\end{definition}

The uniqueness of $\varphi$ is fundamental and it is achieved in the context of weakly Mal'tsev categories, a notion introduced in \cite{NMF.08a}. See also \cite{NMF.08}.

\begin{definition} A finitely complete category $\C$ is weakly Mal'tsev if the morphisms $e_1=\langle 1_A,sf\rangle$ and $e_2=\langle rq,1_C\rangle$ are jointly epimorphic.
\end{definition}

In a weakly Mal'tsev category, the morphism $\varphi$ is unique and so the admissibility of diagram $(\ref{admissibility diagram})$ is a condition and not an additional structure.

We recall that a finitely complete category is Mal'tsev if $(e_1,e_2)$ is a jointly strongly epimorphic pair  (\cite{DB.96}). Hence, all Mal'tsev categories are weakly Mal'tsev but the converse is false. For example, the category of distributive lattices is a weakly Mal'tsev category which is not Mal'tsev (\cite{NMF.12}).

In a weakly Mal'tsev category $\C$ the admissibility of diagrams like $(\ref{admissibility diagram})$ describes several conditions and properties of $\C$. For example, the reflexive graph
$$\xymatrix{ C_1 \ar@<-4pt>[r]_{c} \ar@<4pt>[r]^d & C_0 \ar[l]|e },\quad de=1_{C_0}=ce,$$
is an internal category in $\C$ if and only if the diagram

$$\vcenter{\xymatrix@=4em{C_1 \ar@<.5ex>[r]^-{d} \ar@{=}[rd] & C_0
\ar@<.5ex>[l]^-{e}
\ar@<-.5ex>[r]_-{e}
\ar[d]^-{e} & C_1 \ar@<-.5ex>[l]_-{c} \ar@{=}[ld]\\
& C_1}}$$
is admissible. Recall that  an internal category in $\C$ is a reflexive graph equipped with a morphism $m$ giving composition of arrows
 $$\xymatrix{C_2= C_1\times_{C_0} C_1 \ar[r]^-m  & C_1 \ar@<-4pt>[r]_{c} \ar@<4pt>[r]^d & C_0 \ar[l]|e }$$
  and satisfying associativity and identity axioms. It is an internal groupoid in $\C$ if (in addition) there exists a morphism $i\colon{C_1\to C_1}$ giving inverses for the composition $m$ (see e.g. \cite{BB.04}).

As proved in \cite{NMF.08}, an internal category is a groupoid if and only if the diagram
$$\vcenter{\xymatrix@=4em{C_2 \ar@<.5ex>[r]^-{m} \ar[rd]_-{\pi_1} & C_1
\ar@<.5ex>[l]^-{\langle 1_{C_1},ed\rangle}
\ar@<-.5ex>[r]_-{\langle ec,1_{C_1}\rangle}
\ar@{=}[d] & C_2 \ar@<-.5ex>[l]_-{m} \ar[ld]^-{\pi_2}\\
& C_1}}$$
is admissible. Further details on these characterizations can be found for example in \cite{NMF.VdL.14}.

A pair $(R,S)$ of equivalence relations on an object $X$
$$\vcenter{\xymatrix@=4em{R \ar@<1.0ex>[r]^-{r_1} \ar@<-1.0ex>[r]_-{r_2}  & X
\ar@<.0ex>[l]|-{i_R}
\ar@<-.0ex>[r]|-{i_S}
 & S \ar@<1.0ex>[l]^-{s_2} \ar@<-1.0ex>[l]_-{s_1} }}$$
commute in the sense of Smith-Pedicchio \cite{S.76,MCP.95}, if and only if the diagram
\[
\xymatrix@=4em{R \ar@<.5ex>[r]^-{r_2} \ar[rd]_-{r_1} & X
\ar@<.5ex>[l]^-{i_R}
\ar@<-.5ex>[r]_-{i_S}
\ar@{=}[d] & S \ar@<-.5ex>[l]_-{s_1} \ar[ld]^-{s_2}\\
& X}
\]
is admissible, that is to say if and only if there exists a morphism $\varphi\colon{R\times_{X}S \to X}$ such that $\varphi\langle 1_R, i_Sr_2\rangle=r_1$ and $\varphi\langle i_rs_1,1_S\rangle=s_2$.

A cospan
 \begin{equation*}
\xymatrix@=4em{K \ar@<.0ex>[r]^-{k}  & X
 & L \ar@<-.0ex>[l]_-{l} \\
 }
\end{equation*}
 of morphisms in a pointed category commutes, in the sense of Huq (\cite{SAH.68}), if and only if the diagram

 $$\vcenter{\xymatrix@=4em{K \ar@<.5ex>[r]^-{} \ar[rd]_-{k} & 0
\ar@<.5ex>[l]^-{}
\ar@<-.5ex>[r]_-{}
\ar[d]^-{} & L \ar@<-.5ex>[l]_-{} \ar[ld]^-{l}\\
& X}}$$
 is admissible. That is, there exists a unique morphism $\varphi:K\times L\rightarrow X$ such that $\varphi\langle 1,0\rangle=k$ and $\varphi\langle 0,1\rangle=l$.

When a pointed weakly Mal'tsev category exhibits a coincidence between the Smith-Pedicchio and the Huq commutators then we say that it satisfies the ``Smith is Huq" condition. This phenomenon is analysed in Section \ref{sec:local to global}.

 \vspace{2mm}

 In \cite{L.16} it is observed that the notion of admissibility of diagrams is indeed a commutativity condition.

 \vspace{2mm}

\subsection{Schreier split epimorphisms}

 It is well-known that in the category of groups there is an equivalence between the category of split epimorphisms with codomain $B$
 and kernel $X$ and the one of group actions of $B$ on $X$, that is group homomorphisms $\varphi\colon{B\rightarrow Aut(X)}$.

 This is not true in the category of monoids. There, the classical notion of monoid action, that is a monoid homomorphism $\varphi\colon{B\rightarrow End(X)}$, corresponds to a special class of split epimorphisms, called Schreier split epimorphisms, introduced in \cite{FMS.13}, in the more general context of monoids with operations, and studied in detail in \cite{BFMS.14} and \cite{BFMS.13}. The definition below was inspired by the one of Schreier internal categories given in \cite{P.98} by A. Patchkoria, in the category of monoids.

 \begin{definition}\label{Def:Shreier split epi} A Schreier split epimorphism in the category $Mon$ of monoids is a split epimorphism $(f,r)$, $fr=1_B$, with kernel $X$, for which there exists a unique set-theoretical map $q:A\rightarrow X$, called the Schreier retraction,

 $$\xymatrix{ X \ar@<-2pt>[r]_k  & A \ar@<-2pt>@{.>}[l]_{q}
\ar@<-2pt>[r]_f & B \ar@<-2pt>[l]_r }$$

 such that, for every $a\in A$, $a=kq(a)+rf(a)$.
 \end{definition}

 Equivalently, the following conditions should be satisfied

 \begin{enumerate}[(i)]
 \item  $a=kq(a)+rf(a)$ and
 \item $x=q(k(x)+r(b))$,
 \end{enumerate}
for all $a\in A$, $x\in X$ and $b\in B$, since (ii) gives the uniqueness of the set map $q$.

In the following proposition (proved in \cite{BFMS.13}) we list consequences of the definition that  will be used in the sequel.

\begin{proposition}\label{prop:2.4} Given a Schreier split epimorphism in the category $Mon$ of monoids
 $$\xymatrix{ X \ar@<-2pt>[r]_k  & A \ar@<-2pt>@{.>}[l]_{q}
\ar@<-2pt>[r]_f & B \ar@<-2pt>[l]_r }$$
we have that, for $a,a' \in A$, $x\in X$ and $b\in B$,
\begin{enumerate}[(a)]
  \item $qk=1_X$;
  \item $qr=0$;
  \item $q(0)=0$;
  \item $kq(r(b)+k(x))+r(b)=r(b)+k(x)$;
  \item $q(a+a')= q(a)+q(rf(a)+kq(a'))$;
  \item the sequence is exact, that is $k=\ker f$ and $f=\coker k$, and so we speak of Schreier split extension.

\end{enumerate}

\end{proposition}

To the Schreier split extension above corresponds an action $$\varphi\colon{B\to End(X)}$$ defined by
$$ \varphi(b)(x)=q(r(b)+k(x))$$
that we will denote by $b\cdot x$. Then, for example, condition (\textit{d}) can be written as
$$k(b\cdot x)+r(b)=r(b)+k(x).$$

In Section \ref{sec:5} we will describe Schreier split epimorphisms in the category of cancellative conjugation monoids.

\section{Cancellative conjugation semigroups}

In this section we introduce a subvariety of the variety of unary semigroups consisting of all semigroups equipped with a unary operation that we call conjugation. The main examples that illustrate our intuition behind the notion of conjugation are the semigroups of non-zero complex numbers and non-zero quaternions with the usual multiplication and conjugation. The operation of conjugation is considered for example in \cite{Por.69} as an anti-involution with certain desirable properties, which suggests its study as a unary operation. This study is also inspired by, but it is distinct from, the case of monoids with operations, since the unary operation does not preserve the semigroup addition (see condition (3), below).

\begin{definition}\label{def:3.1}
A conjugation semigroup $(S,+,\overline{( \; )})$ is a semigroup $(S,+)$ equipped with a unary operation $\overline{()}:S\rightarrow S$ satisfying the following identities:

\begin{enumerate}[(1)]
\item $\overline{x}+x=x+\overline{x}$
\item $x+\overline{y}+y=y+\overline{y}+x $
\item $\overline{(x+y)}=\overline{y}+\overline{x}$
\end{enumerate}

\end{definition}

Examples are groups with $\overline{x}=-x$, commutative monoids with $\overline{x}=0$ and commutative semigroups with $\overline{x}=x$.

The quasivariety $\mathcal{S}$ of conjugation semigroups satisfying the implications
\begin{enumerate}[(4)]
\item $x+\overline{a}+a=y+\overline{a}+a\Rightarrow x=y$
\end{enumerate}
is a weakly Malt'tsev category. Indeed, this subvariety has a ternary operation
$$p(x,y,z)=x+\overline{y}+z$$
satisfying the identity
$$p(x,y,y)=p(y,y,x)$$
and the quasi-identity
$$p(x,y,y)=p(x',y,y)\Rightarrow x=x'$$
and so it is weakly Mal'tsev (\cite{NMF.12}).

It is easy to prove that, in presence of (1)-(3), condition (4) is equivalent to cancellation. So, throughout, $\mathcal{S}$ will
denote the \emph{category of cancellative conjugation semigroups}. In the same way,  $\mathcal{M}$ will denote the category of \emph{cancellative conjugation monoids} (which is a non-full subcategory of $\mathcal{S}$).

Since the rules $\overline{x}=x$, $\overline{x}=0$, and $\overline{x}=-x$, define a conjugation on every commutative semigroup, commutative monoid, and group, respectively, and these conjugations are preserved by the homomorphisms of the corresponding categories, then
\begin{itemize}
\item  the category of cancellative commutative semigroups is isomorphic to a full subcategory of $\mathcal{S}$,
\item the category of cancellative commutative monoids is isomorphic to at least two full subcategories of $\mathcal{M}$ (where conjugations are defined by $\overline{x}=x$ and $\overline{x}=0$),
\item the category of groups is isomorphic to a full subcategory of $\mathcal{M}$ (where conjugation is defined by $\overline{x}=-x$),
\end{itemize}

 In abelian groups each of the three rules defines a conjugation. Moreover, it is clear that conjugations on cancellative commutative semigroups, or cancellative commutative monoids, are just the endomorphisms. Indeed, in the monoid case $\overline{0}=0$, which follows from cancellation and the fact that $\overline{0}=\overline{0+0}=\overline{0}+\overline{0}$.

The following are examples of cancellative conjugation semigroups that are neither groups nor monoids:

\begin{enumerate}[1.]
\item $S=\{u\in \mathbb{R} \mid 0<|u|<1\}$ with usual product and $\overline{u}=u$.

\item $S=\{z\in \mathbb{C} \mid 0<\|z\|<1\}$ with usual product and conjugation.

\item $S=\{q\in \mathbb{H} \mid 0<\|q\|<1\}$ with quaternion product and conjugation.
\end{enumerate}

The previous examples are all special cases of the following more general construction.  Let $K$ be a field and suppose that $E$ is a vector space over $K$ in which there is defined, for every $u,v\in E$, a bilinear scalar product $u\cdot v\in K$ and a bilinear vector product $u\times v\in E$. If, in addition, the following two conditions are satisfied for all $u,v,w\in E$
\begin{eqnarray}
u\times(v\times w)-u(v\cdot w)=(u\times v)\times w - (u\cdot v)w\\
u\cdot(v\times w)=(u\times v)\cdot w
\end{eqnarray}
then the set $G=K\times E$ is a conjugation semigroup. Conjugation is defined as $\overline{(\alpha,u)}=(\alpha,-u)$ while the semigroup operation is given by
\[(\alpha,u)\oplus(\beta,v)=(\alpha\beta-u\cdot v,\alpha v+\beta u + u\times v)\]
 for every $\alpha,\beta\in K$ and $u,v\in E$. If, instead of $K$, we take $K^{*}$, the multiplicative group of non-zero scalars, then we obtain a cancellative conjugation semigroup, which is also a monoid, with $(1,0)$ as its identity element. The three examples shown above are particular examples of this more general construction with $K$ the field of real numbers and $E$ the euclidean space (with the usual scalar product) in dimension zero, one and three (with the usual vector product), and its appropriate subsets of unit length vectors.

 \vspace{2mm}

Cancellative semigroups are close to being groups. Indeed,
\begin{itemize}
\item finite semigroups with cancellation are groups;
\item a commutative semigroup can be embedded in an abelian group (i.e. it is isomorphic to a subsemigroup of a group) if and only if it is cancellative;
\item for non-commutative semigroups, cancellability is a necessary but not sufficient condition for embeddability in a group (see \cite{CP.61} and references there);
\end{itemize}

A semigroup $S$ satisfies the Ore condition if $aS \cap bS$ is not empty for all $a, b \in S$.

It is a well known result proved by  Ore in 1931  (\cite{OO.31}; another proof may be found in \cite{CP.61}, Chapter 1) that a cancellative semigroup satisfying the Ore condition can be embedded in a group.

 By 3.1(2), conjugation semigroups satisfy the Ore condition and so every cancellative conjugation semigroup  is  embeddable in a group.

\begin{remark} \emph{ From the definition of conjugation semigroup we obtain the equalities:}

\begin{enumerate}[(i)]
\item $\overline{(x+y)}+(x+y)=\overline{y}+y+x+\overline{x}$

\noindent{\emph{In fact} $\overline{(x+y)}+(x+y)=\overline{y}+\overline{x}+x+y=\overline{y}+x+\overline{x}+y=\overline{y}+y+\overline{x}+x=\overline{y}
+y+x+\overline{x}$.}

\item $x+y+\overline{y}=\overline{y}+y+x$

\end{enumerate}
\emph{that we will use in the sequel.}

\end{remark}

\section{Admissibility in $\mathcal{S}$}

In the weakly Mal'tsev category $\mathcal{S}$ of cancellative conjugation semigroups and semigroup homomorphisms preserving conjugation,
we are going to give a characterization of all admissible diagrams.

\begin{theorem}\label{thm:4.1}
 A diagram
$$
\vcenter{\xymatrix@=4em{A \ar@<.5ex>[r]^-{f} \ar[rd]_-{\alpha} & B
\ar@<.5ex>[l]^-{r}
\ar@<-.5ex>[r]_-{s}
\ar[d]^-{\beta} & C \ar@<-.5ex>[l]_-{g} \ar[ld]^-{\gamma}\\
& D}}
$$
in $\mathcal{S}$, with $fr=qs=1_B$ and $\alpha r=\beta =\gamma s$, is admissible if and only if the following conditions hold:

(Ad1) the equation
$$x+\overline{\beta(b)}+\beta(b)=\alpha(a)+\overline{\beta(b)}+\gamma(c)$$
has a unique solution  for all $a\in A$ and $c\in C$ such that $f(a)=g(c)=b\in B$.

(Ad2) the equation
$$\alpha(a_1+a_2)+\overline{\beta(b_1+b_2)}+\gamma(c_1+c_2)=\alpha(a_1)+\overline{\beta(b_1)}+\gamma(c_1)+\alpha(a_2)+\overline{\beta(b_2)}+\gamma(c_2)$$
 is satisfied for  $a_1,a_2\in A$ and $c_1,c_2\in C$ such that $f(a_1)=g(c_1)=b_1\in B$ and $f(a_2)=g(c_2)=b_2\in B$.

\end{theorem}

\begin{proof}
If there exists a morphism
$\varphi\colon{A\times_{B}C \to D}$ in $\mathcal{S}$ such that $\varphi e_1=\alpha$ and $\varphi e_2=\gamma$ then
$$\alpha(a)=\varphi e_1(a)=\varphi(a,sf(a))=\varphi(a,s(b))$$
$$\gamma(c)=\varphi e_2(c)=\varphi(rg(c),c)=\varphi(r(b),c)$$
and  $\beta(b)=\varphi(r(b),s(b))$, for $f(a)=g(c)=b$.

Then  $\varphi(a,c)$ is a solution of $(Ad1)$:

$\varphi(a,c)+\overline{\beta(b)}+\beta(b)=\varphi(a,c)+\overline{\varphi(r(b),s(
b))}+ \varphi(r(b),s(b))$

$=\varphi(a,c)+\varphi(\overline{r(b)},\overline{s(b)} )+ \varphi(r(b),s(b))$

$=\varphi(a+\overline{r(b)}+r(b),c+\overline{s(b)}+s(b))=\varphi(a+\overline{r(b)}+r(b),s(b)+\overline{s(b)}+c)$

$=\varphi(a,s(b))+\varphi(\overline{r(b)},\overline{s(b)})+\varphi(r(b),c) =\alpha(a)+\overline{\beta(b)}+\gamma(c)
$

because $\varphi$ is a morphism of $\mathcal{S}$ and using Definition 3.1(2).

To prove $(Ad2)$ we use  the previous  result and Remark 3.2:

$\alpha(a_1+a_2)+\overline{\beta(b_1+b_2)}+\gamma(c_1+c_2)=$

$=\varphi(a_1+a_2,c_1+c_2)+\overline{\beta(b_1+b_2)}+\beta(b_1+b_2)$

$=\varphi(a_1,c_1)+\varphi(a_2,c_2)+\overline{(\beta(b_1)+\beta(b_2))}+\beta(b_1)+\beta(b_2)$

$
=\varphi(a_1,c_1)+\varphi(a_2,c_2)+\overline{\beta(b_2)}+\beta(b_2)+\beta(b_1)+\overline{\beta(b_1)}$

$=\varphi(a_1,c_1)+\overline{\beta(b_1)}+\beta(b_1)+\varphi(a_2,c_2)+\overline{\beta(b_2)}+\beta(b_2)$

$=\alpha(a_1)+\overline{\beta(b_1)}+\gamma(c_1)+\alpha(a_2)+\overline{\beta(b_2)}+\gamma(c_2)$

Conversely, let us assume that $(Ad1)$ and $(Ad2)$ hold.   Then we can define $\varphi(a,c)=x$ where $x$ is the  solution of the equation $(Ad1)$, such that $\varphi e_1=\alpha$ and $\varphi e_2=\gamma$. Indeed, $\varphi e_1(a)=\alpha(a)$, since $\varphi e_1(a)=\varphi(a,sf(a))$ satisfies the equation $(Ad1)$:

$\varphi(a,sf(a))+\overline{\beta f(a)}+ \beta f(a)=\alpha(a)+\overline{\beta f(a)}+\gamma sf(a)=$

$=\alpha(a)+\overline{\beta f(a)}+\beta f(a)$.

 By cancellation, we have that $\varphi(a,sf(a))=\varphi e_1(a)=\alpha (a)$.

Analogously,

$\varphi(rg(c),c)+\overline{\beta g(c)}+\beta g(c)=\alpha rg(c)+\overline{\beta g(c)}+\gamma(c)=$

$=\beta g(c)+\overline{\beta g(c)}+\gamma(c)=\gamma(c)+\overline{\beta g(c)}+\beta g(c)$.

So, $\varphi e_2(c)=\varphi(rg(c),c)=\gamma(c)$.

It remains to prove that $\varphi$ is a semigroup homomorphism that preserves  conjugation.

$\varphi(a_1,c_1)+\varphi(a_2,c_2)+\overline{\beta(b_1+b_2)}+\beta(b_1+b_2)=$

$=\varphi(a_1,c_1)+\varphi(a_2,c_2)+\overline{\beta(b_1)+\beta(b_2)}+\beta(b_1)+\beta(b_2)$

$=\varphi(a_1,c_1)+\varphi(a_2,c_2)+\overline{\beta(b_2)}+\beta(b_2)+\beta(b_1)+\overline{\beta(b_1)} $

$=\varphi(a_1,c_1)+\overline{\beta(b_1)}+\beta(b_1)+\varphi(a_2,c_2)+\overline{\beta(b_2)}+\beta(b_2)$

$=\alpha(a_1)+\overline{\beta(b_1)}+\gamma(c_1)+\alpha(a_2)+\overline{\beta(b_2)}+\gamma(c_2)$

$=\varphi(a_1+a_2,c_1+c_2)+\overline{\beta(b_1+b_2)}+\beta(b_1+b_2)$,

for $f(a_1)=g(c_1)=b_1$ and $f(a_2)=g(c_2)=b_2$.

Thus $\varphi(a_1+a_2,c_1+c_2)=\varphi(a_1,c_1)+\varphi(a_2,c_2)$.

Now we have that

$\varphi(\overline{a},\overline{c})+\overline{\beta(\overline{b})}+\beta(\overline{b})=\alpha(\overline{a})+\overline{\beta(\overline{b})}+\gamma(\overline{c})$,

but

$\overline{[\varphi(a,c)]}+\overline{\beta(\overline{b})}+\beta(\overline{b})= \overline{(\beta(b)+\beta(\overline{b})+\varphi(a,c))}
=\overline{(\varphi(a,c)+\beta(\overline{b})+\beta(b))}=$

$=\overline{(\alpha(a)+\beta(\overline{b})+\gamma(c))}
=\gamma(\overline{c})+\overline{\beta(\overline{b})}+\alpha(\overline{a})$,

being $\alpha(a)+\overline{\beta(b)}+\gamma(c)=\gamma(c)+\overline{\beta(b)}+\alpha(a)$, for all $a,b$ and $c$ such that $f(a)=b=g(c)$, that we prove next, and so $\overline{\varphi(a,c)}=\varphi(\overline{a},\overline{c})$.

Indeed,

$\alpha(a)+\overline{\beta(b)}+\gamma(c)=\varphi(a,c)+\overline{\beta(b)}+\gamma(c)=$

$=\varphi(a,c)+\varphi(\overline{r(b)},\overline{s(b)})+\varphi(r(b),s(b))$

$=\varphi(a+\overline{r(b)}+r(b),c+\overline{s(b})+s(b))$

$=\varphi(r(b)+\overline{r(b)}+a,c+\overline{s(b)}+s(b))$

$=\varphi(r(b),c)+\varphi(\overline{r(b)},\overline{s(b}))+\varphi(a,s(b))$

$=\gamma(c)+\overline{\beta(b)}+\alpha(a)$,

and this concludes the proof.
\end{proof}

\section{Schreier split epimorphisms and external actions in $\mathcal{M}$}\label{sec:5}

 In this section we  consider the category $\mathcal{M}$ of cancellative conjugation monoids and monoid homomorphisms that preserve conjugation. This category may be seen as a subcategory of $\mathcal{S}$ and it is immediate to conclude that it is also weakly Mal'tsev. Moreover, the characterization given in Theorem \ref{thm:4.1}  is still valid for diagrams   $(\ref{admissibility diagram})$ in $\mathcal{M}$. To prove that, it is enough to show that if $(Ad1)$ and $(Ad2)$ hold then $\varphi(0,0)=0$.
This indeed is the case, since if

 $\varphi(0,0)+\overline{\beta(0)}+\beta(0)=\alpha(0)+\overline{\beta(0)}+\beta(0)$ in $\mathcal{S}$ then

 $\varphi(0,0)+\overline{\beta(0)}=0+\overline{\beta(0)}$

 because $\alpha(0)=\beta(0)=0$ and so, by cancellation $\varphi(0,0)=0$.

 \hfill

The definition  of Schreier split epimorphism in monoids (Definition \ref{Def:Shreier split epi}) can easily be extended  to the category $\mathcal{M}$. In this context the Schreier retraction $q$, displayed in the diagram
 $$\xymatrix{ X \ar@<-2pt>[r]_k  & A \ar@<-2pt>@{.>}[l]_{q}
\ar@<-2pt>[r]_f & B \ar@<-2pt>[l]_r }$$
interacts with conjugation in the following way:

 \begin{lemma} Given a Schreier split epimorphism in $\mathcal{M}$, the Schreier retraction satisfies the equality

 $$q(\overline{a})=f(\overline{a})\cdot \overline{q(a)}.$$
\end{lemma}

 \begin{proof} We have that

 $f(\overline{a})\cdot \overline{q(a)}=q(rf(\overline{a})+k\overline{q(a)}=q(\overline{kq(a)+rf(a)})=q(\overline{a})$.

  \end{proof}

 So, together with Proposition \ref{prop:2.4} (a)-(e), this describes the behaviour of the Schreier retraction in $\mathcal{M}$.





Let $B$ and $X$ be two cancellative conjugation monoids. An external action of $B$ on $X$ is a classical action at the level of monoids, that is, a monoid homomorphism $\varphi\colon{B\to \text{End($X$)}}$ from the underlying monoid of $B$ into the monoid of endomorphisms of the underlying monoid of $X$, that we write as $\varphi(b)(x)=b\cdot x$.

In the category of monoids, for each Schreier split epimorphism $(f, r)$ with kernel $X$ and Schreier retraction $q$, there exists a commutative diagram
\[
\xymatrix{X\ar[r]^k \ar@{=}[d] & A\ar@<.5ex>[r]^{f}\ar@<-.5ex>[d]_{\alpha} & B\ar@<.5ex>[l]^{r}\ar@{=}[d]\\X \ar[r]_(.35){\langle 1,0\rangle} & X\rtimes_{\varphi} B \ar@<-.5ex>[u]_{\beta} \ar@<.5ex>[r]^(.6){\pi_B}  & B\ar@<.5ex>[l]^(.35){\langle 0,1 \rangle}}
\]
 where $\alpha$ is an isomorphism with inverse $\beta$ defined by $\alpha(a) = (q(a), f(a))$ and $\beta ((x, b)) =k(x) + r(b)$. Using this isomorphism, if $(f, r)$ is a Schreier split epimorphism in $\mathcal{M}$ we conclude that conjugation in the semidirect product $X \rtimes_{\varphi} B$ is given by
 $$\overline{(x, b)} = ( \overline{b}\cdot \overline{x}, \overline{b}).$$

 The question now is, for $B$ and $X$ in $\mathcal{M}$, given an external action of the underlying monoid of $B$ on the underlying monoid of $X$, $\varphi\colon{B\to \text{End($X$)}}$,
 when does the semidirect product defined above is also in $\mathcal{M}$.

 We observe that cancellation immediately follows from cancellation in $X$ and $B$. Indeed, if
 $(x, b) +(u, v) = (x', b') +(u, v)$, that is, if $(x + b \cdot u, b + v) = (x' + b' \cdot u, b' + v)$,
  cancellation in $B$ gives $b = b'$ and then cancellation in $X$ gives $x = x'$.

  We have to find equivalent conditions to the ones
  of the definition of conjugation monoid in terms of the action $\varphi$.

  It is easy to prove that (1), (2) and (3) in Definition \ref{def:3.1} are equivalent to each one of
  the following
three conditions, respectively,
\begin{eqnarray}
\overline{b}\cdot(\overline{x}+x)=x+(b+\overline{b})\cdot\overline{x}\\
x_1+(b_1+\overline{b_1})\cdot (\overline{x_1}+x_2)=x_2+(b_2+\overline{b_1})\cdot(\overline{x_1}+x_1)\\
(\overline{b_2}+\overline{b_1})\cdot(\overline{b_1\cdot x_2})=\overline{b_2}\cdot\overline{x_2}
\end{eqnarray}

This identifies the external actions in $\mathcal{M}$ for which we get the desired equivalence:

\begin{proposition}\label{prop:5.2} The category of external actions in $\mathcal{M}$ is equivalent to the category of Schreier split epimorphisms in $\mathcal{M}$.
 \end{proposition}

In the last section we illustrate the above conditions by constructing a semidirect product in the case of unit quaternions where conjugation coincides with inverses. This gives rise to the usual action by conjugation on the set of quaternions with length less or equal to one, but excluding zero in order to ensure cancellation. The same example is then used to construct an internal category which is not an internal groupoid in $\mathcal{M}$.

\section{Internal structures in $\mathcal{M}$}

In this section we are going to investigate which $\mathcal{M}$-morphisms $h\colon{X\rightarrow B}$, together with a Schreier split epimorphism $(f,r)$
 $$\xymatrix{ X \ar@<-2pt>[r]_k  & A \ar@<-2pt>@{.>}[l]_{q}
\ar@<-2pt>[r]_f & B \ar@<-2pt>[l]_r },$$
 with associated action $\varphi\colon{B\to End(X)}$, induce in $\mathcal{M}$:
\begin{enumerate}[(i)]
\item an internal reflexive graph (in which case $(\varphi, h)$ is said to be a precrossed semimodule),
\item an internal category (in which case $(\varphi, h)$ is said to be a crossed semimodule) or
\item an internal groupoid (then $(\varphi, h)$ is said to be a crossed module),
\end{enumerate}
in the sense we make precise below.

Given $h:X\rightarrow B$ and a Schreier split epimorphism $(f,r)$ in $\mathcal{M}$

 $$\xymatrix{ X \ar@<-2pt>[r]_k \ar@/^1pc/[rr]^{h}  & A \ar@<-2pt>@{.>}[l]|-{q}
\ar@<-2pt>[r]_f & B \ar@<-2pt>[l]|-r }$$
when does $h$ induce a reflexive graph in $\mathcal{M}$? That is, when does it induce an $\mathcal{M}$-morphism $\tilde{h}:A\rightarrow B$
such that

 \begin{center}
 $\tilde{h}k=h$ and $\tilde{h}r=1_B$?
 \end{center}

 If such an $\tilde{h}$ exists, since it is a morphism in $\mathcal{M}$ and $a=kq(a)+rf(a)$, we have that

 \begin{enumerate}[(1)]
 \item $\tilde{h}(a)=\tilde{h}(kq(a)+rf(a))=hq(a)+f(a)$;
 \item $\tilde{h}(\overline{a})=\overline{\tilde{h}(a)}$.
 \end{enumerate}

 \begin{remark}
\emph{   If $\tilde{h}$ exists it is defined by (1) and then (2) holds because }
   $$\tilde{h}(\overline{a})=\tilde{h}(\overline{kq(a)+rf(a)})=\tilde{h}(rf(\overline{a})+k\overline{q(a)})=f(\overline{a})+h\overline{q(a)}$$
 \emph{  and}
$$\overline{\tilde{h}(a)}=\overline{hq(a)+f(a)}=f(\overline{a})+h\overline{q(a)}.$$
 \end{remark}

 \begin{proposition}\label{prop:5.3} In the category $\mathcal{M}$, given a Schreier split epimorphism
  $$\xymatrix{ X \ar@<-2pt>[r]_k  & A \ar@<-2pt>@{.>}[l]_{q}
\ar@<-2pt>[r]_f & B \ar@<-2pt>[l]_r }$$
  a morphism $h:X\rightarrow B$ induces a reflexive graph
 $$\xymatrix{ A \ar@<-4pt>[r]_{\tilde{h}} \ar@<4pt>[r]^f & B \ar[l]|r }$$
if and only if it satisfies the condition

\begin{enumerate}[(C1)]
    \item $h(b\cdot x)+b=b+h(x)$, for all $b \in B$ and $x \in X$.
\end{enumerate}
\end{proposition}

 \begin{proof}
   If there exists $\tilde{h}$, such that $\tilde{h}k=h$ and $\tilde{h}r=1_B$ then, for $b=f(a)$ and $x=q(a)$, we have that

  $  \begin{array}{lll}
      b+h(x) & = & f(a)+hq(a) \\
       & = & \tilde{h}rf(a)+\tilde{h}kq(a)  \\
       & = & \tilde{h}(rf(a)+kq(a)) \\
       & = & \tilde{h}(kq(rf(a)+kq(a))+rf(rf(a)+kq(a))) \\
       & = & \tilde{h}kq(rf(a)+kq(a))+\tilde{h}rf(a) \\
       & = & hq(rf(a)+kq(a))+f(a) \\
       & = & h(b\cdot x)+b \\
    \end{array}$

Conversely, if $h$ satisfies $(C1)$ then we have to prove that the map defined by
$\tilde{h}(a)=hq(a)+f(a)$  preserves addition and conjugation.

\hfill

We have that

  \hfill

 $  \begin{array}{lll}
     \tilde{h}(a+a') & = & hq(a+a')+f(a+a') \\
      & = & h(q(a)+q(rf(a)+kq(a')))+f(a)+f(a'), \emph{(by \ 2.4 (e))}  \\
      & = & hq(a)+hq(rf(a)+kq(a'))+f(a)+f(a') \\
      & = & hq(a)+h(f(a)\cdot q(a'))+f(a)+f(a') \\
   \end{array}
$

 and

  \hfill

$
   \begin{array}{lll}
    \tilde{ h}(a)+\tilde{h}(a') & = & hq(a)+f(a)+hq(a')+f(a') \\
     \emph{(by \ (C1))} & = & hq(a)+h(f(a)\cdot q(a'))+f(a)+f(a'). \\
   \end{array}
 $

 \hfill

  \hfill

 Then also $\overline{\tilde{h}(a)}=\tilde{h}(\overline{a})$ (see Remark 6.1).
 \end{proof}

 We observe that   $\xymatrix{ A \ar@<-4pt>[r]_{\tilde{h}} \ar@<4pt>[r]^f & B \ar[l]|r }$ is called a Schreier reflexive graph in \cite{BFMS.13}, because $(f,r)$ is a Schreier split epimorphism.

 By Proposition 3.2.3 in \cite{BFMS.13}, a Schreier reflexive graph in the category $Mon$ of monoids,
$$\vcenter{\xymatrix@=4em{X \ar@<0.0ex>[r]_-{k}   & X_1
\ar@{-->}@<-1.0ex>[l]_-{q}
\ar@<1.0ex>[r]^-{d} \ar@<-1.0ex>[r]_-{c}
 & X_0 \ar@<0.0ex>[l]|-{s}  }}$$
 with $ds=cs=1_{X_0}$ and $(d,s)$ a Schreier split epimorphism in $Mon$, is an internal category if and only if the condition
 $$ (C) \ q(sc(x_1)+x)+x_1=x_1+x, $$
 for $x_1\in X_1$ and $x\in X=\ker(d)$, is satisfied. In this case the multiplication
$$\xymatrix{ X_1\times_{X_0} X_1 \ar[r]^-m  & X_1 }$$
is given by $m(x_1,x'_1)=kq(x_1)+x'_1$.

\hfill

Let us translate these conditions in terms of the Schreier reflexive graph

$$\xymatrix{ A \ar@<-4pt>[r]_{\tilde{h}} \ar@<4pt>[r]^f & B \ar[l]|r }$$

considering $f$ the domain and $\tilde{h}$ the codomain.

Then
 $$ \xymatrix{ X \ar[r]_(.4){k} & A \ar@<1ex>[r]^(.5){f} \ar@<-1ex>[r]_(.5){\tilde{h}} & B \ar@<-0pt>[l]|(.5){r} }$$
 is an internal category in $Mon$ if and only if the condition
 \begin{enumerate}
 \item[(C)]  $kq(r\tilde{h}(a)+ k(x))+a=a+k(x)$
\end{enumerate}
 is satisfied for all $x\in X$ and $a\in A$. That is
 $$k(\tilde{h}(a)\cdot x)+a=a+k(x).$$
 Then for $f(a)=\tilde{h}(a')=hq(a')+f(a')$,
 $$m(a,a')=kq(a)+a',$$
 where $m:A\times_BA\rightarrow A$ is the monoid homomorphism that gives the admissibility of the diagram
$${\xymatrix@!0@=4em{A \ar@<.5ex>[r]^-{f} \ar@{=}[rd] & B
\ar@<.5ex>[l]^-{r}
\ar@<-.5ex>[r]_-{r}
\ar[d]^-{r} & A \ar@<-.5ex>[l]_-{\tilde{h}} \ar@{=}[ld]\\
& A}} $$
in the category of monoids.

The fact that $me_1(a)=m(a,rf(a))=a$ implies that

$$k(x)=m(k(x),rfk(x))=m(k(x),0).$$

Also $me_2(a)=m(r\tilde{h}(a),a)=a.$

Then, for $(a,a')$ such that $f(a)=\tilde{h}(a'),$

$
  \begin{array}{lll}
    m(a,a') & = & m(kq(a)+rf(a),a') \\
     & = & m((kq(a),0)+(rf(a),a')) \\
     & = & m(kq(a),0)+m(r\tilde{h}(a'),a') \\
     & = & kq(a)+a' .\\
  \end{array}
$

So a monoid homomorphism $m$ satisfying the prescribed conditions of admissibility has to be defined by $m(a,a')=kq(a)+a'$, for all $(a,a')\in A\times_BA$, and it is clear that $m(0,0)=kq(0)+0=0+0=0$.

\hfill

Let us show that the existence of such a morphism $m$ of monoids implies condition $(C)$: since

$$(r\tilde{h}(a),a)+(x,0)=(r\tilde{h}(a)+x,a)$$

then

$
  \begin{array}{lll}
    kq(r\tilde{h}(a)+k(x))+a & = & m(r\tilde{h}(a)+k(x),a) \\
     & = & m(r\tilde{h}(a),a)+m(k(x),0) \\
     & = & a+k(x) \\
  \end{array}
$

Conversely, if condition $(C)$ holds then

$
  \begin{array}{lll}
    m((a_1,a'_1)+(a_2,a'_2)) & = & m(a_1+a_2, a'_1+a'_2) \\
     & = & kq(a_1+a_2)+a'_1+a'_2. \\
  \end{array}
$

\hfill
$
  \begin{array}{lll}
    m(a_1,a'_1)+m(a_2,a'_2) & = & kq(a_1)+a'_1+kq(a_2)+a'_2 \\
     & = & kq(a_1)+k(r\tilde{h}(a'_1)\cdot q(a_2))+a'_1+a'_2, (by \ (C)) \\
     & = & k(q(a_1)+(rf(a_1)\cdot q(a_2))+a'_1+a'_2\\
     & = & kq(a_1+a_2)+a'_1+a'_2, (by \ 2.4 (e)) \\
  \end{array}
$

and so $m$ preserves addition.

\begin{proposition}\label{prop:5.4}
  In the category $\mathcal{M}$, given a Schreier split epimorphism
  $$\xymatrix{ X \ar@<-2pt>[r]_k  & A \ar@<-2pt>@{.>}[l]_{q}
\ar@<-2pt>[r]_f & B \ar@<-2pt>[l]_r }$$
a morphism $h:X\rightarrow B$ induces an internal category if and only if the following conditions hold:

 $(C_1) \ h(b\cdot x)+b=b+h(x), \ for \ all \ x\in X, b\in B$

 $(C_2) \ h(y)\cdot x+y=y+x,  \ for \ all \ x,y\in X$.

\end{proposition}

\begin{proof}
  We first verify that $(C_2)\Leftrightarrow (C)$, so that, in presence of $(C_1)$, we can conclude that the reflexive graph

  $$\xymatrix{ A \ar@<-4pt>[r]_{\tilde{h}} \ar@<4pt>[r]^f & B \ar[l]|r }$$

  is an internal category in $Mon$.

  $(C)\Rightarrow (C_2)$

  If in $k(\tilde{h}(a)\cdot x)+a=a+k(x)$ we take $a=k(y)$, we get

  $$k(\tilde{h}(k(y))\cdot x)+k(y)=k(y)+k(x)$$

  and so, from

  $$k(\tilde{h}k(y)\cdot x+y)=k(y+x),$$

  we conclude that $\tilde{h}k(y)\cdot x+y=y+x,$ because $k$ is injective.

    $(C_2)\Rightarrow (C)$

  $
      \begin{array}{lll}
        k(\tilde{h}(a)\cdot x) +a & = & k( \tilde{h}(kq(a)+rf(a))\cdot x) +kq(a)+rf(a)\\
         & = & k(hq(a)+f(a))\cdot x) +kq(a)+rf(a) \\
         & = & k( hq(a)\cdot (f(a)\cdot x))+kq(a)+rf(a) \ (by \ (C_2)) \\
         & = & kq(a)+k(f(a)\cdot x)+rf(a) \\
         & = & kq(a)+kq(rf(a)+k(x)) +rf(a)\\
         & = & kq(a)+rf(a)+k(x), \ (by \ 2.4 (d)) \\
         & = & a+k(x). \\
      \end{array}$

  Consequently, if $(C_1)$ and $(C_2)$ are satisfied then $h$ induces an internal category in the category $Mon$ of monoids. And in $\mathcal{M}$? It remains to check that $m$ preserves conjugation: for $(a,a') \in A\times_BA$, $m(\overline{a},\overline{a'})=kq(\overline{a})+ \overline{a'}$ and

  \hfill

$
    \begin{array}{lll}
     \overline{ m(a,a')} & = & \overline{kq(a)+a'} \\
       & = & \overline{a'}+k\overline{q(a)} \\
       & = & kq(r\tilde{h}(\overline{a'})+k\overline{q(a)})+\overline{a'}, \ (by \ (C)) \\
       & = & kq(rf(\overline{a})+k\overline{q(a)})+ \overline{a'} \\
       & = & k(f(\overline{a})\cdot \overline{q(a)})+ \overline{a'} \\
       & = & kq(\overline{a})+ \overline{a'}, \ (by \ Lemma \ 5.1) \\
    \end{array}
 $

\end{proof}

\begin{proposition}\label{prop:5.5}
   In the category $\mathcal{M}$, given a Schreier split epimorphism
$$\xymatrix{ X \ar@<-2pt>[r]_k  & A \ar@<-2pt>@{.>}[l]_{q}
\ar@<-2pt>[r]_f & B \ar@<-2pt>[l]_r }$$
a morphism $h:X\rightarrow B$ induces an internal groupoid if and only if the following conditions hold:

 $(C_1) \ h(b\cdot x)+b=b+h(x), \ for \ all \ x\in X, b\in B$

 $(C_2) \ h(y)\cdot x+y=y+x, \ for \ all \ x,y\in X$

 $(C_3) \ X \ is \ a \ group \ and \ -\overline{x}=\overline{(-x)}.$

\end{proposition}

\begin{proof}
  By $(C_1)$ we have a reflexive graph in $\mathcal{M}$
   $$\xymatrix{ A \ar@<-4pt>[r]_{\tilde{h}} \ar@<4pt>[r]^f & B \ar[l]|r }$$
   with $\tilde{h}(a)=hq(a)+f(a)$. Then condition $(C_2)$ is equivalent to the fact that this reflexive graph
   is an internal category:
   $$\xymatrix{ A\times_{B} A \ar[r]^-m  & A  \ar@<-4pt>[r]_{\tilde{h}} \ar@<4pt>[r]^f & B \ar[l]|r }$$
   with $m(a,a')=kq(a)+a'$, for $f(a) =dom(a)$, the domain of a, and $\tilde{h}(a')= cod(a')$, the codomain of $a'$, and so $m(a,a')=a\circ a'.$

   By Proposition 3.3.2 in \cite{BFMS.13}, we know that this Schreier category is a groupoid in $Mon$ if and only if $X$ is a group.

   \hfill

   Let us analyse how the inverses are defined in the ``object of morphisms" $A$. For
   {\small \[\xymatrix{A\times_{B} A \ar[r]^(.65){m} & A
\ar@(ur,ul)[]_{t}
 \ar@<1ex>[r]^{f} \ar@<-1ex>[r]_{\tilde{h}} & B \ar[l]|{r} }
\]}
$t(a)$ is the inverse of $a$, i.e.
      \[ \xymatrix{  \dom(a)=f(a) \ar@<-2pt>[r]_(.5){a} & \cod(a)=\tilde{h}(a) \ar@<-2pt>[l]_(.5){t(a)} }
\]
and
  $t(a)\circ a=m(t(a),a)=1_{f(a)}=rf(a)$
  ,
  $a\circ t(a)=m(a,t(a))=1_{\tilde{h}(a)}=r\tilde{h}(a)$.

  We define $t$ by $t(a)=-kq(a)+r\tilde{h}(a)$. It has the right domain

$
    \begin{array}{lll}
      \dom(t(a)) & = & ft(a) \\
       & = & f(-kq(a)+r\tilde{h}(a)) \\
       & = & -fkq(a)+fr\tilde{h}(a) \\
       & = & 0+\tilde{h}(a)\\
       & = & \tilde{h}(a) \\
    \end{array}
$

and codomain

$
  \begin{array}{lll}
    \cod(t(a)) & = & \tilde{h}(-kq(a)+r\tilde{h}(a)) \\
     & = & -\tilde{h}kq(a)+\tilde{h}r\tilde{h}(a) \\
     & = & -hq(a)+\tilde{h}(a) \\
     & = & -hq(a)+hq(a)+f(a) \\
     & = & f(a). \\
  \end{array}
$

Now we prove that $t(a)$ is the inverse of $a$:

$
  \begin{array}{lll}
    m(a,t(a)) & = & m(a,-kq(a)+r\tilde{h}(a)) \\
     & = & kq(a)-kq(a)+r\tilde{h}(a) \\
     & = & r\tilde{h}(a)=1_{\tilde{h}(a)} \\
  \end{array}
$

and

$
  \begin{array}{lll}
    m(t(a),a) & = & m(-kq(a)+r\tilde{h}(a),a) \\
     & = & kq(-kq(a)+r\tilde{h}(a))+a \\
     & = & kq(k(-q(a))+r(hq(a)+f(a)))+a \\
     & = & k(-q(a))+a \\
     & = & -kq(a)+kq(a)+rf(a) \\
     & = & rf(a)=1_{f(a)} \\
  \end{array}
$

By 3.3.2 in \cite{BFMS.13}, we know that $t$ is a monoid homomorphism. We have to prove that $t$ preserves conjugation, so that it is a morphism of $\mathcal{M}$.

\hfill

We have that

$
  \begin{array}{lll}
    t(\overline{a}) & = & -kq(\overline{a})+r\tilde{h}(\overline{a}) \\
     & = & -kq(\overline{a})+r(hq(\overline{a})+f(\overline{a})) \\
     & = & -k(f(\overline{a})\cdot \overline{q(a)})+r(h(f(\overline{a})\cdot \overline{q(a)})+f(\overline{a})) \ (by \ Lemma \ 5.1) \\
     & = & k(f(\overline{a})\cdot (-\overline{q(a)})+r(f(\overline{a})+h\overline{q(a)}) \ (by \ (C_1)) \\
  \end{array}
$

\hfill

$
  \begin{array}{lll}
    \overline{t(a)} & = & \overline{-kq(a)+r(hq(a)+f(a))} \\
     & = & r(\overline{hq(a)+f(a)})+k(\overline{-q(a)}) \\
     & = & kq(r(\overline{hq(a)+f(a)}) +k(\overline{-q(a)}))+rf(r(\overline{hq(a)+f(a)})+k(\overline{-q(a)}))\\
     & = & k(\overline{hq(a)+f(a)})\cdot (\overline{-q(a)})+r(f(\overline{a})+h\overline{q(a)})) \\
     & = & k(f(\overline{a})\cdot (h\overline{q(a)}\cdot (\overline{-q(a)}))+r(f(\overline{a})+h\overline{q(a)})) \\
     & = & k(f(\overline{a})\cdot (\overline{-q(a)}))+r(f(\overline{a})+h\overline{q(a)}) \\
  \end{array}
$

because $h(\overline{x})\cdot (\overline{-x})=(\overline{-x})$.

Indeed,

 $ \begin{array}{lll}
    0 & = & \overline{x} +(-\overline{x})\\
     & = & \overline{x}+(\overline{-x})  \ (by \ (C_3))\\
     & = & h(\overline{x})\cdot (\overline{-x})+\overline{x} \ (by \ (C_2)) \\
  \end{array}
$

and, since $X$ is a group, $h(\overline{x})\cdot (\overline{-x})=-\overline{x}=(\overline{-x})$.

\hfill

Conversely, if
{\small \[\xymatrix{A\times_{B} A \ar[r]^(.65){m} & A
\ar@(ur,ul)[]_{t}
 \ar@<1ex>[r]^{f} \ar@<-1ex>[r]_{\tilde{h}} & B \ar[l]|{r} }
\]}
is a groupoid in $\mathcal{M}$, and so in $Mon$, we know that $X=\ker (f)$ is a group and that $t$ is defined by

$$t(a)=-kq(a)+r\tilde{h}(a)=-kq(a)+r(hq(a)+f(a)).$$

\hfill

Since $t(\overline{a})=\overline{t(a)}$, in particular, when $a=k(x)$,

\hfill

$t(\overline{k(x)})=-k(\overline{x})+rh(\overline{x}),$

\hfill

$\overline{t(k(x))}=k(\overline{-x})+rh(\overline{x})$

and $t(\overline{k(x)})=\overline{t(k(x))}$ implies that $-k(\overline{x})=k(\overline{-x})$, that is $k(-\overline{x})=k(\overline{-x})$ and so $-\overline{x}=\overline{-x}.$
\end{proof}

\section{From local to global}\label{sec:local to global}

To study internal structures in several contexts (protomodular, homological, semi-abelian categories, see \cite{BB.04}) as well as to obtain there strong properties (see \cite{FL.12} and \cite{FL.15}) it is fundamental that they satisfy the so-called ``Smith is Huq" condition: any two equivalent relations on an object centralize each other, or commute, in the sense of Smith-Pedicchio (see \cite{S.76} and \cite{MCP.95}), if and only if their normalizations commute  in the sense of Huq (\cite{SAH.68}).

Our goal is to show that the category $\mathcal{M}$ of cancellative conjugation monoids satisfies the
``Smith is Huq" condition with respect to Schreier equivalence relations, that is equivalence relations
$$\xymatrix{ R \ar@<-1ex>[r]_{r_2} \ar@<1ex>[r]^{r_1} & X \ar[l]|{i_R} }$$
where $(r_1,i_R)$, and consequently $(r_2,i_R)$, is a Schreier split epimorphism.

\begin{proposition}
  Consider the following diagram in $\mathcal{M}$
  \[ \xymatrix{ & & Y \ar@<-2pt>[d]_l & \\
& A \times_B C \ar@<-2pt>[d]_{p_1} \ar@<-2pt>[r]_-{p_2}
& C \ar@<-2pt>[l]_-{e_2} \ar@<-2pt>[d]_g \ar@<-2pt>@{.>}[u]_{q_g} \ar@/^/[ddr]^{\gamma} \\
X \ar@<-2pt>[r]_k & A \ar@<-2pt>[u]_{e_1} \ar@<-2pt>[r]_f
\ar@<-2pt>@{.>}[l]_{q_f} \ar@/_/[drr]_{\alpha} & B \ar@<-2pt>[l]_r
\ar@<-2pt>[u]_s \ar[dr]^{\beta} & \\
& & & D, } \]
with $(f,r)$ and $(g,s)$ Schreier split epimorphisms with kernels $X$ an $Y$, respectively, and $\alpha r=\beta=\gamma s$.

If $\alpha k$ and $\gamma l$ commute then there exists a unique morphism $\varphi\colon A\times_BC\rightarrow D$ such that $\varphi e_1=\alpha$ and $\varphi e_2=\gamma$.
\end{proposition}
\begin{proof}
  This relies on Proposition 6.2 in \cite{FM.17} where the result was proved for monoids with operations.
  This is not the case here because the unary operation given by the conjugation does not preserve
  addition. The morphisms $\alpha k$ and $\gamma l$ Huq-commute if, for all $x\in X$ and $y\in Y$,
$$\alpha k(x)+\gamma l(y)=\gamma l(y)+\alpha k(x).$$

  Then the morphism $\varphi:A\times_BC\rightarrow D$ is defined by
  $$\varphi(a,c)=\alpha kq_f(a)+\gamma(c),$$
  for all $a\in A$ and $c\in C$ such that $f(a)=g(c)=b$.
  Indeed,
$$
    \begin{array}{lll}
      \varphi(a,c) & = & \varphi(kq_f(a)+rf(a),c) \\
       & = & \varphi(kq_f(a),0)+\varphi(rf(a),c) \\
       & = & \varphi(kq_f(a),0)+\varphi(rg(c),c) \\
       & = & \alpha kq_f(a)+\gamma(c) \\
    \end{array}
 $$
 is such that $\varphi e_1=\alpha$ and $\varphi e_2=\gamma$. Furthermore, it was proved in 6.2 \cite{FM.17} that $ \varphi$ is a monoid homomorphism.
 It remains to prove that $\varphi$ preserves conjugation.

 For $(a,c)\in A\times_BC$,

$
   \begin{array}{lll}
     \varphi(\overline{a},\overline{c}) & = & \alpha kq_f(\overline{a})+\gamma(\overline{c}) \\
      & = & \alpha kq_f(\overline{kq_f(a)+rf(a)})+\gamma(\overline{c}) \\
      & = & \alpha kq_f(rf(\overline{a})+k\overline{q_f(a)})+\gamma(lq_g(\overline{c})+sg(\overline{c})) \\
      & = & \alpha kq_f(rf(\overline{a})+k\overline{q_f(a)})+\gamma lq_g(\overline{c})+\gamma sg(\overline{c}) \\
      & = & \gamma lq_g(\overline{c})+\alpha kq_f(rf(\overline{a})+k\overline{q_f(a)})+\alpha rf(\overline{a}),  (\ \gamma l \ and \ \alpha k \ commute)\\
      & = & \gamma lq_g(\overline{c})+\alpha(kq_f(rf(\overline{a})+k\overline{q_f(a)})+rf(\overline{a})) \\
      & = & \gamma lq_g(\overline{c})+\alpha(rf(\overline{a})+k\overline{q_f(a)}), \ (by \ \mbox{Prop.} 2.4 \ (d))\\
      & = & \gamma lq_g(\overline{c})+\alpha rf(\overline{a})+\alpha k\overline{q_f(a)} \\
      & = &  \gamma lq_g(\overline{c})+\gamma sg(\overline{c})+\alpha k\overline{q_f(a)}\\
      & = & \gamma(lq_g(\overline{c})+sg(\overline{c}))+\alpha k\overline{q_f(a)}, \ (\mbox{by Def. 2.3}) \\
      & = & \gamma(\overline{c})+ \alpha k\overline{q_f(a)}\\
   \end{array}
$

 and

 \hfill

$
   \begin{array}{lll}
    \overline{ \varphi(a,c)} & = & \overline{\alpha kq_f(a)+\gamma(c)} \\
      & = & \gamma(\overline{c})+ \alpha k\overline{q_f(a)}. \\
   \end{array}
 $

\end{proof}

Given two equivalence relations $R$ and $S$ on an object $X$
$$\vcenter{\xymatrix@=4em{R \ar@<1.0ex>[r]^-{r_1} \ar@<-1.0ex>[r]_-{r_2}  & X
\ar@<.0ex>[l]|-{i_R}
\ar@<-.0ex>[r]|-{i_S}
 & S \ar@<1.0ex>[l]^-{s_2} \ar@<-1.0ex>[l]_-{s_1} }},$$
they commute if and only the diagram
$$\vcenter{\xymatrix@=4em{R \ar@<.5ex>[r]^-{r_2} \ar[rd]_-{r_1} & X
\ar@<.5ex>[l]^-{i_R}
\ar@<-.5ex>[r]_-{i_S}
\ar@{=}[d] & S \ar@<-.5ex>[l]_-{s_1} \ar[ld]^-{s_2}\\
& X}}$$
is admissible. Then if $(r_2,i_R)$ and $(s_1,i_S)$ are Schreier split epimorphisms, by Proposition 7.1. we conclude that if $r_1 \ker (r_2)$ and $s_2 \ker (s_1)$ commute in the sense of Huq then the Schreier equivalence relations $R$ and $S$ commute. Since the converse is always true in a weakly Mal'tsev category, we conclude that $\mathcal{M}$ satisfies the following relative ``Smith is Huq" property:

\begin{theorem}
  In the category $\mathcal{M}$ of cancellative conjugation monoids two Schreier equivalence relations on an object commute if and only if their normalizations commute.
\end{theorem}

We consider now the case where just one of the split epimorphisms is a Schreier split epimorphism. The reason why we are interested in this case is to be able to arive at Proposition \ref{prop:7.4} and obtain an alternative process of showing that internal categories are the same as crossed semimodules.

\begin{proposition}\label{prop:6.3}
  Let $(f,r)$ be a Schreier split epimorphism with kernel $k$ and retraction $q$ in the category $\mathcal{M}$. Then for the diagram
  \[ \xymatrix{ & &  & \\
& A \times_B C \ar@<-2pt>[d]_{p_1} \ar@<-2pt>[r]_-{p_2}
& C \ar@<-2pt>[l]_-{e_2} \ar@<-2pt>[d]_g  \ar@/^/[ddr]^{\gamma} \\
X \ar@<-2pt>[r]_k & A \ar@<-2pt>[u]_{e_1} \ar@<-2pt>[r]_f
\ar@<-2pt>@{.>}[l]_{q} \ar@/_/[drr]_{\alpha} & B \ar@<-2pt>[l]_r
\ar@<-2pt>[u]_s \ar[dr]^{\beta} & \\
& & & D, } \]
with $fr=gs=1_B$, $\alpha r=\gamma s=\beta$,  the following conditions are equivalent:

\begin{enumerate}[(i)]
  \item There exists a morphism $\varphi:A\times_BC\rightarrow D$ such that $\varphi e_1=\alpha$ and $\varphi e_2=\gamma$.
  \item There exists a morphism $\varphi:A\times_BC\rightarrow D$ such that $\varphi \langle k,0\rangle=\alpha k$ and $\varphi e_2=\gamma$.
  \item For all $x\in X$ and $c\in C$, $\alpha k(g(c)\cdot x)+\gamma(c)=\gamma(c)+\alpha k(x)$.
\end{enumerate}
\end{proposition}

\begin{proof}
$(i)\Rightarrow (ii)$ is obvious. To prove the converse we observe that since $(f,r)$ is a Schreier split epimorphism then $(k,r)$ is a jointly strongly epimorphic pair (\cite{BFMS.13}, 2.1.6) and so, since $\varphi e_1k=\varphi\langle k,0>=\alpha k$ and $\varphi e_1r=\varphi e_2s=\gamma s=\alpha r$ then $\varphi e_1=\alpha$. And $\varphi e_2=\gamma$.

The morphism $\varphi$, if it exists, has to be defined by $$\varphi(a,c)=\alpha kq(a)+\gamma(c)$$

because

\hfill

$
  \begin{array}{lll}
    \varphi(a,c) & = & \varphi(kq(a)+rf(a),c) \\
     & = & \varphi(kq(a),0)+\varphi(rf(a),c) \\
     & = & \varphi(kq(a),0)+\varphi(rg(c),c) \\
     & = & \alpha kq(a)+\gamma (c). \\
  \end{array}
$

Then
$\varphi(a_1,c_1)+\varphi(a_2,c_2)=\alpha kq(a_1)+\gamma(c_1)+\alpha kq(a_2)+\gamma(c_2)$
and
$
  \begin{array}{lll}
    \varphi(a_1+a_2,c_1+c_2)  & = & \alpha k(q(a_1)+q(rf(a_1)+kq(a_2)))+\gamma(c_1+c_2) \\
     & = & \alpha kq(a_1)+\alpha kq(rg(c_1)+kq(a_2))+\gamma(c_1)+\gamma(c_2). \\
  \end{array}
$
Thus $\varphi(a_1+a_2,c_1+c_2)=\varphi(a_1,c_1)+\varphi(a_2,c_2)$ if and only if
$$\alpha kq(rg(c_1)+kq(a_2))+\gamma(c_1)=\gamma(c_1)+\alpha kq(a_2),$$
since $ q(rg(c1) + kq(a_2) = g(c_1) \cdot  q(a_2)$,
that is if and only if $(iii)$ holds.

\hfill

We have that $(iii)\Leftrightarrow (i)$ because $\varphi$ also preserves conjugation:

the fact that $\overline{\varphi(a,c)}=\varphi(\overline{a},\overline{c})$ is equivalent to the identity

$$\alpha k(\overline{f(a)}\cdot \overline{q(a)})+\gamma(\overline{c})=\gamma(\overline{c})+\alpha k\overline{q(a)},$$
thanks to Lemma 5.1,
that is to $(iii)$ for $f(\overline{a})=g(\overline{c})$ and $x=\overline{q(a)}$.
\end{proof}

Finally, if $A=C$ and $r=s$ we are in the case considered in Proposition \ref{prop:5.3} (where $g$ was denoted by $\tilde{h}$) that is we have a reflexive graph
$$\vcenter{\xymatrix@=4em{X \ar@<0.0ex>[r]_-{k}   & X_1
\ar@{-->}@<-1.0ex>[l]_-{q}
\ar@<1.0ex>[r]^-{f} \ar@<-1.0ex>[r]_-{g}
 & X_0 \ar@<0.0ex>[l]|-{r}  }}$$
with $(f,r)$ a Schreier split epimorphism, that is induced by $h=gk$.

Then, taking $ c=k(y)$ in $(iii)$ we obtain
$$(iii)' \ \alpha k(h(y)\cdot x)+\gamma k(y)=\gamma k(y)+\alpha k(x)$$
for all $x,y\in X$ and so $\alpha k$ and $\gamma k$ ``Huq-commute" up to the action of $h(y)$ on $x$.

Conversely, if $(iii)'$ holds then, since $c=kq(c)+fr(c)$,

$
   \begin{array}{lll}
     \alpha k(g(c)\cdot x)+\gamma(c) & = & \alpha k(g(kq(c)+fr(c))\cdot x)+\gamma(kq(c)+fr(c)) \\
      & = & \alpha k(hq(c)\cdot (f(c)\cdot x))+\gamma kq(c)+\gamma rf(c) \\
      & = & \gamma k(q(c))+\alpha k(f(c)\cdot x)+\alpha rf(c) \\ 
      & = & \gamma k(q(c))+\alpha(k(f(c)\cdot x)+rf(c)) \\
      & = & \gamma k(q(c))+\alpha(rf(c)+k(x)) \ (by \ 2.4 \ (d)) \\
      & = & \gamma k(q(c))+\gamma rf(c)+\alpha k(x) \ (\alpha r=\gamma r) \\
      & = & \gamma(kq(c)+rf(c))+\alpha k(x) \\
      & = & \gamma(c)+\alpha k(x). \\
   \end{array}
$

\hfill

Thus we proved the following:

\begin{proposition}\label{prop:7.4}
  If, in Proposition \ref{prop:6.3}, $A=C$ and $s=r$ then the diagram
  \begin{equation}\label{diag:1}
   \xymatrix{ & &  & \\
& A \times_B A \ar@<-2pt>[d]_{p_1} \ar@<-2pt>[r]_-{p_2}
& A \ar@<-2pt>[l]_-{e_2} \ar@<-2pt>[d]_g  \ar@/^/[ddr]^{\gamma} \\
X \ar@<-2pt>[r]_k & A \ar@<-2pt>[u]_{e_1} \ar@<-2pt>[r]_f
\ar@<-2pt>@{.>}[l]_{q_f} \ar@/_/[drr]_{\alpha} & B \ar@<-2pt>[l]_r
\ar@<-2pt>[u]_r \ar[dr]^{\beta} & \\
& & & D, } \end{equation}
is admissible if and only if
$$\alpha k(h(y)\cdot x)+\gamma k(y)=\gamma k(y)+\alpha k(x), \ for \  all \ x,y\in X.$$

\end{proposition}

An even more particular case of the previous proposition can be related with the condition for a crossed semimodule thus giving an alternative explanation on how does the local behaviour of the kernel $X$  can be extended to the global behaviour of $A$.

Indeed, if in Proposition \ref{prop:6.3} we take $A=C$, $s=r$, $\alpha=1$, $\beta=r$ and $\gamma=1$ then the admissibility of the diagram (\ref{diag:1})
is equivalent to the condition of $h=g k$ being a crossed semimodule, that is
$$h(y)\cdot x+y=y+x, \ for \  all \ x,y\in X.$$

\section{Examples}

In this section we give a non-trivial example of an internal category which is not an internal groupoid. The example is based on quaternions which have been used as a model for the notion of cancellative conjugation monoid.

The set $\mathbb{H}\setminus \{0\}$ of non-zero quaternions is a non-commutative group for the usual multiplication. It is cancellative and has conjugation: for $q=a+bi+cj+dk$ the conjugate is $\overline{q}=a-bi-cj-dk$.

We recall that the norm of $q$ is given by $\|q\|=\sqrt{q\overline{q}}=\sqrt{a^2+b^2+c^2+d^2}$, it is multiplicative, $\|pq\|=\|p\|\|q\|$, and $q^{-1}=\frac{\overline{q}}{\|q\|^2}$ for $q\not=0$,

\hfill

The sets $B=\{q\in \mathbb{H}\mid\|q\|=1\}$ and $X=\{q\in \mathbb{H}\mid0<\|q\|\leq 1\}$ are a conjugation group ($q^{-1}=\overline{q}$) and a conjugation monoid, respectively, with cancellation.

\hfill

We are going to construct a Schreier split epimorphism in $\mathcal{M}$. For that we consider the monoid action $\varphi$ of $B$ on $X$ defined by $b\cdot x=bxb^{-1}=bx\overline{b}$, the semidirect product $X\rtimes_\varphi B$ that is the monoid with underlying set $X\times B$ and operation
$$(x_1,b_1)(x_2,b_2)=(x_1(b_1\cdot x_2),b_1b_2)$$
and the Schreier split epimorphism of monoids
$$\xymatrix{ X \ar@<-2pt>[r]_-{\langle 1,0\rangle} & X\rtimes_\varphi B \ar@<-2pt>@{.>}[l]_-{\pi_1}
\ar@<-2pt>[r]_-{\pi_2} & B \ar@<-2pt>[l]_-{\langle 0,1\rangle }}$$

As a consequence of Proposition \ref{prop:5.2}, we obtain a Schreier split epimorphism in $\mathcal{M}$  provided we take $\overline{(x,b)}=(\overline{b}\cdot \overline{x},\overline{b})$. It is interesting to observe the details in verifying the axioms for a cancellative conjugation monoids.
Indeed,

\begin{enumerate}
  \item $(x,b)\overline{(x,b)}=\overline{(x,b)}(x,b)$

$
  \begin{array}{lll}
    (x,b)\overline{(x,b)} & = & (x,b)(\overline{b}\cdot \overline{x},\overline{b}) \\
     & = & (x(b\cdot (\overline{b}\cdot \overline{x})),b\overline{b}) \\
     & = & (x(b\overline{b}\cdot \overline{x}),b\overline{b}) \\
     & = & (x\overline{x},1) \ (because \ \overline{b}=b^{-1}) \\
  \end{array}
$

\hfill

$
  \begin{array}{lll}
    \overline{(x,b)}(x,b) & = & (\overline{b}\cdot \overline{x},\overline{b})(x,b) \\
     & = & ((\overline{b}\cdot \overline{x})(\overline{b}\cdot x),\overline{b}b) \\
     & = & (\overline{b}\cdot \overline{x}x,\overline{b}b) \\
     & = & (\overline{x}x,1) \\
  \end{array}
$

because the center of $\mathbb{H}\setminus \{0\}$ is $\mathbb{R}$ and so $\overline{b}(\overline{x}x)\overline{b}^{-1}=\overline{b}(\overline{x}x)b=\overline{b}b(\overline{x}x)=\overline{x}x.$
And $x\overline{x}=\overline{x}x$.

\vspace{.3cm}
\item $(x,b)\overline{(y,c)}(y,c)=(y,c)\overline{(y,c)}(x,b)$

$
    \begin{array}{lll}
      (x,b)\overline{(y,c)}(y,c) & = & (x,b)(\overline{y}y,1) \\
       & = & (x(b\cdot \overline{y}y),b) \\
       & = & (x\overline{y}y,b) \ (because \ \overline{y}y \ is \ in \ the \ center \  of \ \mathbb{H})\\
    \end{array}
 $

 \hfill

$
   \begin{array}{lll}
     (y,c)\overline{(y,c)}(x,b) & = & (y\overline{y},1)(x,b) \\
      & = & (y\overline{y}(1\cdot x),b) \\
      & = & (y\overline{y}x,b) \\
   \end{array}
$

and $x\overline{y}y=y\overline{y}x$ in $\mathbb{H}$.

\vspace{.3cm}

\item  $\overline{(x,b)(y,c)}=\overline{(y,c)} \; \overline{(x,b)}$

  $
    \begin{array}{lll}
      \overline{(x,b)(y,c)} & = & \overline{(x(b\cdot y),bc)} \\
       & = & (\overline{bc}\cdot \overline{(x(b\cdot y))},\overline{bc}) \\
    \end{array}
$

\hfill

$
  \begin{array}{lll}
    \overline{(y,c)} \; \overline{(x,b)} & = & (\overline{c}\cdot \overline{y},\overline{c})(\overline{b}\cdot \overline{x},\overline{b}) \\
     & = & ((\overline{c}\cdot \overline{y})(\overline{c}\cdot (\overline{b}\cdot \overline{x})),\overline{c}\overline{b}) \\
  \end{array}
$

 and

 \hfill

$
  \begin{array}{lll}
 \overline{bc}\cdot \overline{(x(b\cdot y))} & = & \overline{c}\overline{b}\cdot (\overline{(b\cdot y)}\overline{x})\\
      & = & \overline{c}\overline{b}\cdot (\overline{(by\overline{b})}\overline{x}) \\
     & = & \overline{c}\overline{b}\cdot (b\overline{y}\overline{b}\overline{x}) \\
     & = & \overline{c}\overline{b} b \overline{y}\overline{b}\overline{x} bc \ ((\overline{c}\overline{b})^{-1}=bc)\\
      & = & \overline{c}\overline{y}\overline{b}\overline{x}bc  \ (\overline{b}b=1) \\
  \end{array}
$

\hfill

$
  \begin{array}{lll}
    (\overline{c}\cdot \overline{y})(\overline{c}\cdot (\overline{b}\cdot \overline{x})) & = & \overline{c}\cdot (\overline{y}(\overline{b}\cdot \overline{x})) \\
     & = & \overline{c}\cdot (\overline{y}\overline{b}\overline{x}b) \ (since \ b^{-1}=\overline{b}) \\
     & = & \overline{c}\; \overline{y}\overline{b}\overline{x}bc \\
  \end{array}
$

and so (3) holds.
\end{enumerate}

\hfill

So $X\rtimes_\varphi B$ is an $\mathcal{M}$-object and it gives rise to the Schreier split epimorphism in this category associated with the action $\varphi$.

We define $h\colon{X\rightarrow B}$ by $h(x)=\frac{x}{\|x\|}$, which gives a monoid homomorphism (because the norm is multiplicative) that, furthermore, preserves conjugation.

The morphism $h$ induces an internal category in the sense of Proposition \ref{prop:5.4} because it satisfies condition $(C_1)$ and $(C_2)$:

$$h(b\cdot x)b=\frac{bxb^{-1}}{\|x\|}b=b\frac{x}{\|x\|}=bh(x),$$

and

$$(h(y)\cdot x)y=\frac{y}{\|y\|}x\left(\frac{y}{\|y\|}\right)^{-1}y = yx.$$.

But $h$ does not induce an internal groupoid, because $X$ is not a group.

\hfil

Summing up:

\begin{example}\label{eg:8.1}
  Given $X$ and $B$ as above, the action $\varphi$ of $B$ on $X$ defined by $\varphi(b)(x)=bxb^{-1}$ gives rise to a Schreier split epimorphism
  $$\xymatrix{ X \ar@<-2pt>[r]_-{\langle 1,0\rangle} & X\rtimes_\varphi B \ar@<-2pt>@{.>}[l]_-{\pi_1}
\ar@<-2pt>[r]_-{\pi_2} & B \ar@<-2pt>[l]_-{\langle 0,1\rangle }}$$
in $\mathcal{M}$, with $\overline{(x,b)}=(\overline{b}\cdot \overline{x},\overline{b})$ in the semidirect product. Then $h\colon{X\rightarrow B}$, defined by $h(x)=\frac{x}{\|x\|}$, is a morphism in $\mathcal{M}$ that induces an internal category  but not an internal groupoid since $X$ is not a group, in the sense of  Propositions \ref{prop:5.4} and \ref{prop:5.5}, respectively.

\end{example}

In other words we say that the pair $(h,\varphi)$ consisting of the morphism $h$ and the action $\varphi$ defined above is a crossed semimodule but not a crossed module. In order to give an example of a precrossed semimodule which is not a crossed semimodule we can simply take $B = 0$ and $h$ the unique morphism $X \to 0=B$. Indeed, $(h,\varphi)$ satisfies $(C_1)$ but not $(C_2)$ because $X$ is not commutative.

Example \ref{eg:8.1} has a geometrical interpretation. In order to see it we will consider the particular case of complex numbers instead of quaternions. In that case, $B$ is the unit circle and $X$ is the unitary disk with its center removed. The action is trivial and the morphism $h\colon{X\to B}$, defined as $x\mapsto \frac{x}{\|x\|}$, projects a point on the disk, radially, into its boundary. We may now form a category whose objects are the points on the unit circle and with arrows the pairs $(x,b)\in X\times B$. The morphisms in this category have an interesting and intuitive interpretation. Indeed, an arrow in this category, being a pair $(x,b)\in X\times B$, let us say of the form $x=\alpha e^{i\theta_1}$, $b=e^{i\theta_2}$, with $i^2=-1$ the imaginary unit, $\alpha\in]0,1]$ and $\theta_1,\theta_2\in]-\pi,\pi]$, can be interpreted as an arc of circumference (centred at the origin in the complex plane) with radio $\alpha$, starting at angle $\theta_2$  and ending at angle $\theta_1+\theta_2 (mod \ 2\pi)$. If $(x',b')$ and $(x,b)$ are interpreted as two morphisms, then they can be composed when $b'=\frac{x}{\|x\|}b$ and its composition is the pair $(x'x, b)$. It is now clear that not every morphism has an inverse. For example, if $b=1$ and $x=\frac{i}{2}$,  then the inverse to the arrow $(x,b)$ would be $(-2i,i)$, which does not belong to the set $X\times B$.

\section*{Acknowledgements}
 This work was partially supported by Funda\c{c}\~{a}o para a Ci\^{e}ncia e a Tecnologia (FCT) via: (CDRSP--UID/Multi/04044/2019) and (CMUC -- UID/MAT/00324/2019); PAMI - ROTEIRO/0328/2013 (Nº 022158);  Next.parts (17963); Centro2020; CDRSP and ESTG from the Polytechnic Institute of Leiria, Centro de Matem\'{a}tica da Universidade de Coimbra, Faculdade de Ci\^{e}ncias e Tecnologia da Universidade dos A\c{c}ores.

\end{document}